\documentclass[reqno,12pt]{amsart} %
\usepackage{amsmath,amstext, amsthm, amssymb,amsfonts}
\usepackage{fancybox,color}

 \usepackage[top=1 in,bottom=1in, left=1 in, right= 1 in]{geometry}

\numberwithin{equation}{section}

\newtheorem{theorem}{Theorem}[section]
\newtheorem{proposition}[theorem]{Proposition}
\newtheorem{lemma}[theorem]{Lemma}
\newtheorem{corollary}[theorem]{Corollary}
\theoremstyle{definition}

\newcount\mins \newcount\hours \hours=\time \mins=\time
\divide\hours by60 \multiply\hours by 60 \advance\mins by-\hours
\divide\hours by60

\def\now{ \ifnum\hours>11 \ifnum\hours>12 \advance\hours by
-12 \fi \number\hours:\ifnum\mins<10 0\fi \number\mins\ pm,\ \else
\ifnum\hours=0 \hours=12 \fi \number\hours:\ifnum\mins<10 0\fi
\number\mins\ am,\ \fi}

\newcommand{\V}{\mathbb{V}}

\newcommand{\bea}{\begin{equation}}
\newcommand{\eea}{\end{equation}}

\newcommand{\CSK}{Cauchy-Stieltjes Kernel (CSK)\ \renewcommand{\CSK}{CSK\ }}
\newcommand{\FEF}{Free Exponential  (FE)\ \renewcommand{\FEF}{FE\ }}

\title[Characterization of quadratic CSK families by orthogonality of polynomials]{Characterization of quadratic Cauchy-Stieltjes Kernel families by orthogonality of polynomials}

\author{Raouf Fakhfakh}
\address{ Faculty of Sciences, Sfax University, Tunisia. }
\email{fakhfakh.raouf@gmail.com}
\date{Printed \today. File \jobname.tex}
\begin{document}

\begin{abstract}
In this paper we specify some facts about the sequence of polynomials associated to a \CSK family and we prove that quadratic variance function is characterized by the property of orthogonality of these polynomials.
\end{abstract}

\maketitle

\section{Introduction}

The notion of variance function of a natural exponential family (NEF) has drawn considerable attention of recherche and many classifications of NEF by the form of their variance function has been realized. For many common NEFs the variance function takes a very simple form. Morris (\cite{Morris}) describe the class of real NEFs such that the variance function is a polynomial in the mean of degree less than or equal to two. In \cite{Letac-Mora}, Letac and Mora have extended the work of Morris by classifying all real cubic NEFs such that the variance function is a polynomial in the mean of degree less than or equal to three. 
This classes has received a deal of attention in the statistical literature and many interesting characteristic properties have been established.
A remarkable characteristic result is due to Meixner (\cite{Meixner}). It characterizes the distribution $\mu$ for which there exists a family of $\mu$-orthogonal polynomials with an exponential generating function. These distributions generates exactly the Morris class of NEFs. A second characterization is due to Feinsilver (\cite{Feinsilver}), who shows that a certain class of polynomials naturally associated to a NEF is $\mu$-orthogonal if and only if the family is in the Morris class.  In \cite{Hassairi-Zarai}, Hassairi and Zarai have introduced the notion of 2-orthogonality for a sequence of polynomials to give extended versions of the Meixner and Feinsilver characterizations results based on orthogonal polynomials. In fact, they show that the cubicity of the variance function is characterized by the property of 2-orthogonality.

In a manner analogous to the definition of NEFs, Bryc and Ismail (see \cite{Bryc-Ismail-05}) have introduced the definition of $q$-exponential families. They have identified all $q$-exponential families when $|q|<1$. In particular they have studied the case when $q=0$ which is related to free probability theory by using the Cauchy-Stieltjes kernel $1/(1-\theta x)$ instead of the exponential kernel $\exp(\theta x)$. In \cite{Bryc-06-08}, Bryc continue  the study of  \CSK families for compactly supported probability measures $\nu$. He has shown that such families can be parameterized by the mean and under this parametrization, the family (and measure $\nu$) is uniquely determined by the variance function $V(m)$ and the mean $m_0$ of $\nu$. He has also described the class of quadratic \CSK families. Up to affine transformations and powers of free convolution, this class consists of the free Meixner distributions. In \cite{Bryc-Hassairi-09}, Bryc and Hassairi continue the study of \CSK families by extending the results to allow measures $\nu$ with unbounded support, providing the method to determine the domain of means, introducing  the ``pseudo-variance" function that has no direct probabilistic interpretation but  has similar properties to the variance function. They have also introduced the notion of reciprocity between tow \CSK families by defining a relation between the $R$-transforms of the corresponding generating probability measure. This leads to describe a class of cubic \CSK families which is related to quadratic class by a relation of reciprocity.

In this paper, we are interested in the class of quadratic \CSK families: Our aim is to characterize such families by the property of orthogonality of polynomials in the Meixner and Feinsilver way. In section 2, after a review of \CSK families, we specify some facts about the sequence of polynomials associated to a \CSK family, in particular we show that the generating function of this sequence converge in a neighborhood of 0. In section 3, we state and prove our main result concerning the characterization of the sequence of polynomials corresponding to distribution generating a quadratic \CSK family by a property of orthogonality. Then, we  determine the families of orthogonal polynomials with a Cauchy-Stieltjes type generating function. This leads to another characterization of the quadratic \CSK family.
\section{Polynomials associated to  CSK families}
The \CSK families arise from a procedure analogous to the definition of NEFs by using the Cauchy-Stieltjes kernel $1/(1-\theta x)$ instead of the exponential kernel $ \exp(\theta x)$. In this section, we present the basic concept of \CSK families and we define the associated polynomials. We first review some facts concerning the polynomials  associated to NEFs.
\subsection{NEFs and associated polynomials}
If $\mu$ is a positive measure on the real line, we denote by
\begin{equation}\label{laplace transform}
L_\mu(\theta)=\displaystyle\int_{\mathbb{R}}\exp(\theta x)\mu(dx),
\end{equation}
its Laplace transform, and we denote $\Theta(\mu)=\textrm{interior}\{\theta\in \mathbb{R};\ L_\mu(\theta)<\infty\}.$
$\mathcal{M}(\mathbb{R})$ will denote the set of measures $\mu$ such that $\Theta(\mu)$ is not empty and $\mu$ is not concentrated on one point. If $\mu$ is in $\mathcal{M}(\mathbb{R})$, we also denote
\begin{equation}\label{cumulate function}
K_\mu(\theta)=\log(L_\mu(\theta)),\ \ \ \ \ \theta\in\Theta(\mu),
\end{equation}
the cumulate function of $\mu$. To each $\mu$ in $\mathcal{M}(\mathbb{R})$ and $\theta$ in $\Theta(\mu)$, we associate the following probability distribution:
\begin{equation}\label{P-theta}
P(\theta,\mu)(dx)=\exp(\theta x-K_\mu(\theta))\mu(dx).
\end{equation}
The set
\begin{equation}\label{def NEF}
F=F(\mu)=\{P(\theta,\mu),\ \theta\in\Theta(\mu)\}
\end{equation}
is called the natural exponential family (NEF) generated by $\mu$. The map $\theta\longmapsto K'_\mu(\theta)$ is a bijection between $\Theta(\mu)$ and its image $M_F$ which is called the domain of the means of the family F. Denote by $\phi_\mu:M_F\longrightarrow \Theta(\mu)$ the inverse of $K'_\mu$.
We are thus led to the parametrization of $F$ by the mean $m$. For each $\mu\in\mathcal{M}(\mathbb{R})$ and $m\in M_F$, let us denote $P(m,F)=P(\phi_\mu(m),\mu)$
and rewrite $F=\{P(m,F);\ m\in M_F\}$.
The density of $P(m,F)$ with respect to $\mu$ is
\begin{equation}\label{density P(m,F)}
h_{\mu}(x,m)=\exp(\phi_\mu(m) x-k_\mu(\phi_\mu(m))).
\end{equation}
The variance of $P(m,F)$ is denoted $V_F(m)$. The map $m\longmapsto V_F(m)$ is called the variance function of the NEF $F$ and is defined for all $m\in M_F$ by
$$V_F(m)=K''_\mu(\phi_\mu(m))=(\phi'_\mu(m))^{-1}.$$
The important feature of $V_F(.)$ is that it characterizes the NEF F in the following sense: If $F_1$ is another NEF such that $M_F\cap M_{F_1}$ contains a non-empty open interval $\mathcal{O}$ and $V_F(m)=V_{F_1}(m)$ for $m\in\mathcal{O}$, then $F=F_1$. Thus $(M_F,V_F(m))$ completely characterizes $F$.

Consider now a real natural exponential family $F$ and take $\mu=P(m_0,F)$ with $m_0$ fixed in $M_F$. the density $h_{\mu}(.,m)$ of $P(m,F)$ with respect to $\mu$ is still given by \eqref{density P(m,F)} with $h_{\mu}(.,m_0)\equiv 1$. It is easily verified by induction on $\mathbb{N}$ that there exists a polynomials $H_n$ in $x$ of degree $n$ such that
\begin{equation}\label{to get poly}
\frac{\partial^n}{\partial m^n}h_{\mu}(x,m)=H_n(x,m)h_{\mu}(x,m).
\end{equation}
and
\begin{equation}\label{recursion}
H_{n+1}(x,m)=\phi'_{\mu}(m)(x-m)H_{n}(x,m)+R_{n+1}(x,m),
\end{equation}
where $R_{n+1}$ is a polynomial in $x$ of degree $<n+1$. In particular, we have that $H_0(x,m)=1$ and $H_1(x,m)=\phi'_{\mu}(m)(x-m)$.
\subsection{\CSK families and associated polynomials}

 The notations are the ones used in \cite{Bryc-Hassairi-09}. Suppose $\nu$ is a non-degenerate probability measure with support bounded from above. Then
\begin{equation}   \label{M(theta)}
M_{\nu}(\theta)=\int \frac{1}{1- \theta x} \nu(dx)
\end{equation}
 is well defined for all $\theta\in [0,\theta_+)$ with $1/\theta_+=\max\{0, \sup {\rm supp} (\nu)\}$ and
$$P_\theta(dx)=\frac{1}{M_{\nu}(\theta)(1-\theta
x)}\nu(dx)$$
is a probability measure for each $\theta\in[0,\theta_+)$. Then an analog of the NEF, with Cauchy kernel $1/(1-\theta x)$ replacing the exponential kernel $\exp(\theta x)$, is  the family
 \begin{equation}\label{K_and_F}
\mathcal{K}_+(\nu)=\{P_\theta(dx); \theta\in(0,\theta_+)\}=\{Q_m(dx), m\in(m_0,m_+)\}
 \end{equation}
which we call the (one-sided) \CSK family generated by $\nu$.

As in the case of NEF,
the \CSK family can be re-parameterized by the mean, and we already included this alternative parametrization on the right hand side of \eqref{K_and_F}.   
The interval $(m_0,m_+)$ is called the (one sided) domain of means, and is determined as the
image of $(0,\theta_+)$ under the strictly increasing function
$k_{\nu}(\theta)=\int x P_\theta(dx)$ which is given by the formula
\begin{equation}\label{L2m}
k_{\nu}(\theta)=\frac{M_{\nu}(\theta)-1}{\theta M_{\nu}(\theta)}.
\end{equation}
The variance function
\begin{equation}
  \label{Def Var}
  V_{\nu}(m)=\int (x-m)^2 Q_m(dx)
\end{equation}
is the fundamental concept of the theory of NEF, and also of the theory of \CSK families. Unfortunately,
if $\nu$ does not have the first moment, all measures in the \CSK family generated
by $\nu$ have infinite variance. Reference \cite{Bryc-Hassairi-09} introduces the concept of {\em pseudo-variance function} which is defined in general by
\begin{equation}\label{m2v}
\V_{\nu}(m)=m\left(\frac{1}{\psi_{\nu}(m)}-m\right),
\end{equation}
where $\psi_{\nu}:(m_0,m_+)\to (0,\theta_+)$ is the inverse of the
function $k_{\nu}(\cdot)$.
If
$m_0=\int x d\nu$ is finite, the variance function given by  \eqref{Def Var} exists, in fact from proposition
3.2 in \cite{Bryc-Hassairi-09} we know that
\begin{equation}\label{VV2V}
\V_{\nu}(m)=\frac{m}{m-m_0}V_{\nu}(m).\end{equation}
In particular, $\V_{\nu}=V_{\nu}$ when $m_0=0$. Specifically, the re-parameterized measure  involves
the  pseudo-variance function $\V_{\nu}(m)$ and is given by
\begin{equation} \label{F(V)}
 Q_m(dx)=\frac{\V_{\nu}(m)}{\V_{\nu}(m)+m(m-x)}\nu(dx).
\end{equation}


Another interesting fact is that the pseudo-variance function $\V_{\nu}$ characterizes the \CSK family, in fact the generating measure $\nu$ is determined uniquely through the following identities, for technical details, see proposition 3.5 in \cite{Bryc-Hassairi-09}:
if
\begin{equation}\label{z2m}
 z=z(m)=m+\frac{\V_{\nu}(m)}{ m}
 \end{equation}
 then the Cauchy transform
\begin{equation}\label{G-transform}
G_\nu(z)=\int\frac{1}{z-x}\nu(dx).
\end{equation}
satisfies
 \begin{equation}
  \label{G2V}
  G_\nu(z)=\frac{m}{\V_{\nu}(m)}.
\end{equation}

For a non-degenerate probability measure $\nu$ with support bounded
from above, the domain of means $(m_0,m_+)$ is
determined from the following formulas (\cite[Remark 3.3]{Bryc-Hassairi-09})
$m_0:=\lim_{\theta\to 0^+} k_{\nu}(\theta)$
and $m_+=B-1/G_\nu(B)$, with $B=B(\nu)=\max\{0,\ \sup supp(\nu)\}$.

 One may define the one-sided \CSK family for a
generating measure with support bounded from below. The
one-sided \CSK family ${\mathcal{K}}_-(\nu)$ is defined for
$\theta_-<\theta<0$, where $\theta_-$ is either $1/A(\nu)$ or
$-\infty$ with $A=A(\nu)=\min\{0,\inf supp(\nu)\}$.
In this case, the domain of the means for ${\mathcal{K}}_-(\nu)$ is the interval $(m_-,m_0)$ with $m_-=A-1/G_\nu(A)$. If $\nu$ has compact support, the natural domain for the
parameter $\theta$ of the two-sided \CSK family
$\mathcal{K}(\nu)=\mathcal{K}_+(\nu)\cup\mathcal{K}_-(\nu)\cup\{\nu\}$
is $\theta_-<\theta<\theta_+$.

As indicated in formula \eqref{to get poly}, for NEFs the associated polynomials are obtained by taking successive  derivative of the density $h_{\mu}(.,m)$ of $P(m,F)$ with respect to $\mu$. For the \CSK families we use the $q$-derivative for $q=0$. Usually the $q$-derivative operator $\mathcal{D}_q$ is defined by
$$\mathcal{D}_qf(x)=\frac{f(x)-f(qx)}{(1-q)x}, \ \ \ \ \mbox{for}\ x\neq0,$$
where $q$ is fixed and $-1<q<1$. Further we have that $\mathcal{D}^n_qf(x)=\mathcal{D}_q(\mathcal{D}^{n-1}_qf(x))$, for $n=1,2,3,...$, where $\mathcal{D}^0_q$ denotes the identity operator. If $f$ is differentiable then $\mathcal{D}_qf(x)$ tends to $f'(x)$ as $q$ tends to $1$. Note that for any positive integer and $f$ a function for which $f^{(n)}(0)$ exists, we have
$$(\mathcal{D}^n_qf)(0)=\displaystyle\lim_{x\longrightarrow 0}\mathcal{D}^n_qf(x)=\frac{f^{(n)}(0)}{n!}[n]_q!,$$
with $[n]_q!=\displaystyle\frac{(q;q)_n}{(1-q)^n}$ such that for $a\in\mathbb{R}$, $(a;q)_0=1$ and $(a;q)_n=\prod_{k=0}^{n-1}(1-aq^k)$, $n=1,2,...$ In particular, we have that $(\mathcal{D}_0f)(0):=f'(0)$ and for each $n\in\mathbb{N}^*$ for which $f^{(n)}(0)$ exists,
$$(\mathcal{D}^n_0f)(0)=\displaystyle\lim_{x\longrightarrow 0}\mathcal{D}^n_0f(x)=\frac{f^{(n)}(0)}{n!}.$$


\begin{proposition}
Let  ${\mathcal{K}}(\nu)=\{Q_m(dx),\ m\in(m_-,m_+)\}$ be the \CSK family generated by a compactly supported probability measure $\nu$ with mean $m_0(\nu)=0$. Suppose that $\V_{\nu}$ is analytic near $0$ and $\V_{\nu}(0)>0$. The density $f_{\nu}(x,m)$ of $Q_m$ with respect to $\nu$ is given by \eqref{F(V)}. Then there exists a polynomials $P_n$ in $x$ of degree $n$ such that
\begin{equation}\label{poly}
\mathcal{D}_0^nf_{\nu}(x,m)=P_n(x,m)f_{\nu}(x,m).
\end{equation}
In particular $P_0(x,m)=1$ and $P_1(x,m)=\frac{x-m}{\V_{\nu}(m)}$.
\end{proposition}
\begin{proof} We verify this result by induction on $n\in\mathbb{N}$. For $n=0$, we have that $\mathcal{D}^0_0f_{\nu}(x,m)=1\times f_{\nu}(x,m)$, so that $P_0(x,m)=1$. For $n=1$, we have that $$\mathcal{D}_0f_{\nu}(x,m)=\frac{x-m}{\V_{\nu}(m)}f_{\nu}(x,m).$$
Then, $P_1(x,m)=\frac{x-m}{\V_{\nu}(m)}$ is a polynomial in $x$ of degree 1, in particular $P_1(x,0)=x/\V_{\nu}(0).$ For $n=2$, we have
$$
\mathcal{D}_0^2f_{\nu}(x,m)
=\frac{\mathcal{D}_0f_{\nu}(x,m)-\mathcal{D}_0f_{\nu}(x,0)}{m}
=  \left[\frac{x-m}{m\V_{\nu}(m)}-
\frac{x(\V_{\nu}(m)+m(m-x))}{m\V_{\nu}(0)\V_{\nu}(m)}\right]f_{\nu}(x,m).
$$
So, $$P_2(x,m)=\frac{x-m}{m\V_{\nu}(m)}-
\frac{x(\V_{\nu}(m)+m(m-x))}{m\V_{\nu}(0)\V_{\nu}(m)}$$ is a polynomial in $x$ of degree 2, in particular
$$P_2(x,0)=\mathcal{D}_0^2f_{\nu}(x,0)=\displaystyle\lim_{m\longrightarrow 0}\mathcal{D}_0^2f_{\nu}(x,m)=\frac{x^2-x\V'_{\nu}(0)-\V_{\nu}(0)}{\V_{\nu}(0)^2}.$$
 Let $n> 2$, suppose that there exists a polynomials $P_n$ in $x$ of degree $n$ such that
$$\mathcal{D}_0^nf_{\nu}(x,m)=P_n(x,m)f_{\nu}(x,m).$$
In this case we have $$P_n(x,0)=\mathcal{D}_0^nf_{\nu}(x,0)=\frac{1}{n!}\left.\frac{\partial^n}{\partial m^n}f_{\nu}(x,m)\right|_{m=0},$$
which is well defined from the fact that  $\V_{\nu}$ is analytic near $0$ and $\V_{\nu}(0)>0$. We have that
\begin{eqnarray*}
\mathcal{D}_0^{n+1}f_{\nu}(x,m)
& = &\mathcal{D}_0(\mathcal{D}_0^nf_{\nu}(x,m))=\frac{\mathcal{D}_0^nf_{\nu}(x,m)-\mathcal{D}_0^nf_{\nu}(x,0)}{m}\\
& = & \left[\frac{P_n(x,m)}{m}-
\frac{P_n(x,0)(\V_{\nu}(m)+m(m-x))}{m\V_{\nu}(m)}\right]f_{\nu}(x,m)= P_{n+1}(x,m)f_{\nu}(x,m).
\end{eqnarray*}
It is clear that
$P_{n+1}(x,m)=\frac{P_n(x,m)}{m}-
\frac{P_n(x,0)(\V_{\nu}(m)+m(m-x))}{m\V_{\nu}(m)}$ is a polynomial in $x$ of degree $n+1$.


\end{proof}
We now make a useful observation through the following lemma, that will be used in the proof of Theorem \ref{TH 2}.
\begin{lemma}
Let $\nu$ be a compactly supported probability measure such that $0\in\Theta=(\theta_-,\theta_+)$ and $m_0(\nu)=0$. Denote by $(T_n)$ the sequence of orthogonal polynomials with respect to $\nu$ such that $T_n$ is monic of degree $n$. Let $r=\sup\{\alpha;\ (-\alpha,\alpha)\subset\Theta\}$. Then the entire serie $\sum z^nT_n(x)$ valued in $L^2(\nu)$ has radius of convergence $\geq r$.
\end{lemma}
\begin{proof}
We have to prove that for $|z|<r$
$$\lim_{N\longrightarrow +\infty}\int\left(\sum_{n=0}^Nz^nT_n(x)\right)^2\nu(dx)=\lim_{N\longrightarrow +\infty}\sum_{n=0}^Nz^{2n}\int T_n^2(x)\nu(dx)<+\infty.$$
Denote $b_n=\int T_n^2(x)\nu(dx)$ and $\gamma_n=\int x^n\nu(dx)$. We have that
$$\int x^{2n}\nu(dx)=\int (x^{n}-T_n(x))^2\nu(dx)+\int T_n^{2}(x)\nu(dx).$$
That is $\gamma_{2n}\geq b_n$. From the analyticity of $M_{\nu}$ on $\Theta$, if $|z|<r$
$$M_{\nu}(z)=\int\frac{1}{1-zx}\nu(dx)=\int\sum_{n\geq0}z^nx^n\nu(dx)=\sum_{n\geq0}z^n\gamma_n.$$
Hence $$\sum_{n\geq0}z^{2n}b_{n}\leq\sum_{n\geq0}z^{2n}\gamma_{2n}$$
and converge if $|z|<r$.
\end{proof}
To help in the proof of Theorem \ref{TH 1}, we need to state the following result.
\begin{lemma}\label{L.G(Q)}
Let  $m\in(m_-,m_+)$. Then $Q_m$ is a probability measure.
For $z\in\mathbb{C}\backslash supp(\nu)$ such that $ z\neq m+\V_{\nu}(m)/m $ the Cauchy transform of $Q_m\in\mathcal{K}(\nu)$ is
\begin{equation}\label{G(nu)2G(Q)}
G_{Q_m}(z)=
\frac{1}{m+\V_{\nu}(m)/m-z}\left(\frac{\V_{\nu}(m)}{m}G_\nu(z)-1\right).
\end{equation}

\end{lemma}
\begin{proof} From \eqref{F(V)},
\begin{multline*}
\frac{1}{z-x}Q_m(dx)
\\=\frac{\V_{\nu}(m)}{(z-x)(m^2+\V_{\nu}(m)-mz)}\nu(dx) - \frac{ \V_{\nu}(m)}{(\V_{\nu}(m)+m(m-x))((m+\V_{\nu}(m)/m)-z)}\nu(dx)
\\=\frac{\V_{\nu}(m)}{m^2+\V_{\nu}(m)-mz}(z-x)^{-1}\nu(dx) -\frac{m}{m^2+\V_{\nu}(m)-mz}Q_m(dx).
\end{multline*}
 Integrating, we get \eqref{G(nu)2G(Q)}.

\end{proof}

\section{Characterizations of the quadratic CSK families }
As pointed out in the introduction, reference \cite{Bryc-06-08} describe the class of quadratic \CSK families; that is the class of \CSK families such that the corresponding variance function is a polynomial function in the mean of degree at most 2. Up to affine transformation and powers of free convolution, this class consists of, the Wigner's semi-circle (free gaussian), the Marchenko Pastur (free Poisson), the free Pascal (free negative binomial), the free Gamma, the free analog of hyperbolic type law and the free binomial families. In this section, we give new versions of the Feinsilver and Meixener characterizations results based on orthogonal polynomials. These versions subsume the quadratic class of \CSK families.

\subsection{Characterization of the quadratic CSK families in the Feinsilver way}
Feinsilver (\cite{Feinsilver}) characterizes the class of quadratic NEFs on $\mathbb{R}$ as the ones for which the associated polynomials are orthogonal with respect to the generating measure, more precisely, (see \cite{Letac}):
\begin{theorem}\label{TH 01}
Let $F$ be a NEF on $\mathbb{R}$ and let $\mu$ an element of $F$ with mean $m_0$. Consider the polynomials $(H_n)_{n\in\mathbb{N}}$ defined by $H_n(x)=\frac{\partial^n}{\partial m^n}h_{\mu}(x,m)|_{m=m_0}$. Then the following statements are equivalent:
\begin{itemize}
\item[(i)] The polynomials $(H_n)_{n\in \mathbb{N}}$ are $\mu$-orthogonal.
\item[(ii)] $F$ is a quadratic NEF.
\item[(iii)] There exists real numbers $(\alpha_i)_{0\leq i\leq2}$ such that
$$xH_n(x)=n(\alpha_2(n-1)+1)H_{n-1}(x)+(n\alpha_1+m_0)H_n(x)+\alpha_0H_{n+1}(x).$$
Furthermore, in this case we have $V_{F}(m)=\alpha_0+\alpha_1(m-m_0)+\alpha_2(m-m_0)^2.$
\end{itemize}
\end{theorem}

The polynomials associated to quadratic \CSK families have also a characterizing property of orthogonality in the Feinsilver way, more precisely we have
\begin{theorem}\label{TH 1}
Let ${\mathcal{K}}(\nu)=\{Q_m(dx),\ m\in(m_-,m_+)\}$ be the \CSK family generated by a compactly supported probability measure $\nu$ with mean $m_0(\nu)=0$. Suppose that $\V_{\nu}$ is analytic near $0$ and $\V_{\nu}(0)>0$. The density $f_{\nu}(x,m)$ of $Q_m$ with respect to $\nu$ is given by \eqref{F(V)}. Consider the polynomials $(P_n)_{n\in\mathbb{N}}$ defined by
\begin{equation}\label{def poly}
\left.P_n(x)=\displaystyle\lim_{m\longrightarrow 0}\mathcal{D}_0^nf_{\nu}(x,m)
=\frac{1}{n!}\frac{\partial^n}{\partial m^n}f_{\nu}(x,m)\right|_{m=0}.
\end{equation}
Then the three following statements are equivalent:
\begin{itemize}
\item[(i)] The polynomials $(P_n)_{n\in \mathbb{N}}$ are $\nu$-orthogonal.
\item[(ii)] ${\mathcal{K}}(\nu)$ is a quadratic \CSK family.
\item[(iii)] There exists real numbers $(a_i)_{0\leq i\leq2}$ such that
$$xP_n(x)=(1+a_2)P_{n-1}(x)+a_1P_n(x)+a_0P_{n+1}(x).$$
Furthermore, in this case we have $\V_{\nu}(m)=a_0+a_1m+a_2m^2.$
\end{itemize}
\end{theorem}
It is shown in \cite{Favard}, that there exists a unique compactly supported positive measure $\nu$ on $\mathbb{R}$, up to a constant multiplication, such that a sequence of polynomials $(T_n)_{n\in\mathbb{N}}$, generated by a three-terms recursion formula with constant coefficients, are $\nu$-orthogonal. In \cite{Cohen-Trenholme}, Cohen and Trenholme calculated the measure $\nu$ explicitly, for which  the sequence of polynomials $(T_n)$ is orthogonal. The normalization for measure $\nu$, given by Cohen and Trenholme is not one for the probability measure. In (\cite{Saitoh-Yoshida}, Theorem 2.1) a modified version of this result is given, by normalizing their measure, to obtain probability measure. Theorem \ref{TH 1} deals with orthogonal polynomials from a point of view related to \CSK families.
\begin{proof}
$(i)\Longrightarrow (ii)$. There exists $r>0$ such that for all $m$ in $]-r,r[$,
$$f_{\nu}(x,m)=\sum_{n\geq0}m^nP_n(x)$$
If for $(m,\widetilde{m})\in(]-r,r[)^2$, we set
$$g(m,\widetilde{m})=\int f_{\nu}(x,m)f_{\nu}(x,\widetilde{m})\nu(dx)$$
then from the orthogonality of the polynomials $(P_n)$, we get
\begin{eqnarray*}
g(m,\widetilde{m}) & = & \int f_{\nu}(x,m)f_{\nu}(x,\widetilde{m})\nu(dx)
 =  \int\left(\displaystyle\sum_{n,\widetilde{n}\geq0}
m^n(\widetilde{m})^{\widetilde {n}}P_n(x)P_{\widetilde{n}}(x)\right)\nu(dx)\\
& = & \displaystyle\sum_{n\geq0}(m\widetilde{m})^{n}\int P_n(x)^2 \nu(dx)=1+\displaystyle\sum_{n\geq1}(m\widetilde{m})^{n}\int P_n(x)^2 \nu(dx).
\end{eqnarray*}
On the other hand, we have
$$
g(m,\widetilde{m})
 =  \int
\frac{\V_{\nu}(\widetilde{m})}{\V_{\nu}(\widetilde{m})+\widetilde{m}(\widetilde{m}-x)}Q_m(dx)
 =  \frac{\V_{\nu}(\widetilde{m})}{\widetilde{m}}
G_{Q_m}\left(\frac{\V_{\nu}(\widetilde{m})}{\widetilde{m}}+\widetilde{m}\right).
$$
Using \eqref{G(nu)2G(Q)}, we get
\begin{equation}\label{g(m,m tilde)}g(m,\widetilde{m})=\frac{\widetilde{m}\V_{\nu}(m)-m\V_{\nu}(\widetilde{m})}
{\widetilde{m}m^2-m\widetilde{m}^2+\widetilde{m}\V_{\nu}(m)-m\V_{\nu}(\widetilde{m})}.
\end{equation}
Taking the derivative of \eqref{g(m,m tilde)} with respect to $\widetilde{m}$, we get for all $(m,\widetilde{m})\in(]-r,r[)^2$
\begin{equation}\label{derive g}
\frac{m\widetilde{m}^2\V_{\nu}(m)+(m^2\widetilde{m}^2-\widetilde{m}m^3)\V'_{\nu}(\widetilde{m})
+(m^3-2m^2\widetilde{m})\V_{\nu}(\widetilde{m})}
{[\widetilde{m}m^2-m\widetilde{m}^2+\widetilde{m}\V_{\nu}(m)-m\V_{\nu}(\widetilde{m})]^2}=
\displaystyle\sum_{n\geq1}n\beta_n m(m\widetilde{m})^{n-1},
\end{equation}
with $\beta_n=\int P_n(x)^2 \nu(dx)$.\\
Making $\widetilde{m}=0$ in \eqref{derive g}, we get $m/\V_{\nu}(0)=m\beta_1$. This is true for all $m\in ]-r,r[$, then $\beta_1=1/\V_{\nu}(0).$
Again we take the derivative of \eqref{derive g} with respect to $\widetilde{m}$ and we let $\widetilde{m}=0$, we get for all $m$ in $]-r,r[$
$$[\V_{\nu}(m)+m^2-m\V'_{\nu}(0)-\V_{\nu}(0)]/\V_{\nu}(0)^2=\beta_2m^2.$$
Therefore,
$$\V_{\nu}(m)=(\beta_2\V_{\nu}(0)^2-1)m^2+m\V'_{\nu}(0)+\V_{\nu}(0).$$
Then, $\V_{\nu}$ is quadratic on $]-r,r[$, and by extension ${\mathcal{K}}(\nu)$ is a quadratic \CSK family.

$(ii)\Longrightarrow (iii)$. From (ii), there exists real numbers $(a_i)_{0\leq i\leq2}$ such that
$$\V_{\nu}(m)=a_0+a_1m+a_2m^2.$$
On the other hand we know that there exists $r>0$ such that for all $m\in ]-r,r[$,
\begin{equation}\label{f_nu}
f_{\nu}(x,m)=\sum_{n\geq0}m^nP_n(x).
\end{equation}
Applying $\mathcal{D}_0(.) $ to the both side of \eqref{f_nu}, we obtain
$$\frac{x-m}{\V_{\nu}(m)}f_{\nu}(x,m)=\sum_{n\geq0}m^{n-1}P_n(x),$$
which is equivalent to
$$\sum_{n\geq0}(x-m)m^nP_n(x)=\sum_{n\geq0}(a_0+a_1m+a_2m^2)m^{n-1}P_n(x).$$
Then, we have
$$\sum_{n\geq0}(x-a_1)m^nP_n(x)=\sum_{n\geq0}(1+a_2)m^{n+1}P_n(x)+
\sum_{n\geq0}a_0m^{n-1}P_n(x).$$
By identification, we get
$$xP_n(x)=(1+a_2)P_{n-1}(x)+a_1P_n(x)+a_0P_{n+1}(x).$$
$(iii)\Longrightarrow (i)$. The result is easily obtained if we verify the three following facts:

$(a)$ For all $n\in\mathbb{N}^*$, $\int P_n(x)\nu(dx)=0$.

$(b)$ There exists real numbers $\beta^s_{n,q}$ such that, for all $n,\ q\in\mathbb{N}^*$
$$x^q P_n(x)=\beta^0_{n,q}P_{n-q}(x)+\sum_{n-q+1\leq s\leq n+q }\beta^s_{n,q}P_s(x)$$
with $\beta^0_{n,q}=0$ if $n<q$.

$(c)$ There exists real numbers $(\alpha_q)_{0\leq q\leq n}$ such that
$$P_n(x)=\alpha_n x^n+\sum_{0\leq q\leq n-1}\alpha_q x^q.$$

\textit{Proof of} $(a)$. We first observe that
$$\int \mathcal{D}_0f_{\nu}(x,m)\nu(dx)=\int\frac{(x-m)}{\V_{\nu}(m)} f_{\nu}(x,m)\nu(dx)=\frac{1}{\V_{\nu}(m)}\int(x-m)Q_m(dx)=0.$$
We have that
\begin{eqnarray*}
\int \mathcal{D}_0^{n+1}f_{\nu}(x,m)\nu(dx)
& = & \int\frac{\mathcal{D}_0^nf_{\nu}(x,m)-\mathcal{D}_0^nf_{\nu}(x,0)}{m}\nu(dx)\\
& = & \frac{\int\mathcal{D}_0^nf_{\nu}(x,m)\nu(dx)-\int\mathcal{D}_0^nf_{\nu}(x,0)\nu(dx)}{m}=  \mathcal{D}_0\left(\int\mathcal{D}_0^nf_{\nu}(x,m)\nu(dx)\right).
\end{eqnarray*}
Hence we obtain that, for all $n\in\mathbb{N}^*$
$$\int\mathcal{D}_0^nf_{\nu}(x,m)\nu(dx)=0.$$
\begin{eqnarray*}
|\mathcal{D}_0^nf_{\nu}(x,m)| & = & |P_n(x,m)f_{\nu}(x,m)|\leq\sup_{m\in(m_-,m_+)}\sup_{x\in\  supp(\nu)}|P_n(x,m)|f_{\nu}(x,m)\\
& \leq & \sup_{m\in(m_-,m_+)}\sup_{x\in\  supp(\nu)}\{|P_n(x,m)|\}\frac{\V_{\nu}(m)/m}{\V_{\nu}(m)/m+m-x}\\
& \leq & \sup_{m\in(m_-,m_+)}\sup_{x\in\  supp(\nu)}\{|P_n(x,m)|\}\frac{\sup_{m\in(m_-,m_+)}\{\V_{\nu}(m)/m\}}{B(\nu)-x}=g(x).
\end{eqnarray*}
Here we use the fact that $\V_{\nu}(m)/m+m\geq B(\nu)$ for all $m\in(m_-,m_+)$. It is clear that $g(.)$ is $\nu$-integrable. This implies that
\begin{equation}
  \label{Question1}
  \lim_{m\longrightarrow 0}\int \mathcal{D}_0^nf_{\nu}(x,m)\nu(dx)=\int \lim_{m\longrightarrow 0}\mathcal{D}_0^nf_{\nu}(x,m)\nu(dx).
\end{equation}
This end the proof of $(a)$.

\textit{Proof of} $(b)$. We can write (iii) as
 \begin{equation}\label{recurcive relation 1}
x P_n(x)=\beta^0_{n,1}P_{n-1}(x)+\displaystyle\sum_{n\leq s\leq n+1}\beta^s_{n,1}P_{s}(x).
\end{equation}
For a fixed $n\in \mathbb{N}^*$, let us show by induction that for all $q\in \mathbb{N}^*$ such that $q\leq n$, we have
\begin{equation}\label{recurcive relation 2}
x^q P_n(x)=\beta^0_{n,q}P_{n-q}(x)+\displaystyle\sum_{n-q+1\leq s\leq n+q}\beta^s_{n,q}P_{s}(x).
\end{equation}
where $\beta^0_{n,q}=0$ if $n=q$.

For $q=1$, it is nothing but equality \eqref{recurcive relation 1}.

Suppose now that \eqref{recurcive relation 2} is true for $q$ such that $q+1\leq n$. Then we have
\begin{eqnarray*}
x^{q+1}P_n(x)
& = & x(x^q P_n(x))=\beta^0_{n,q}xP_{n-q}(x)+\displaystyle\sum_{n-q+1\leq s\leq n+q}\beta^s_{n,q}xP_{s}(x)\\
& = & \beta^0_{n,q}\left\{\beta^0_{n-q,1}P_{n-q-1}(x)+\displaystyle\sum_{n-q\leq s'\leq n-q+1}\beta^{s'}_{n-q,1}P_{s'}(x)\right\}\\
& + & \displaystyle\sum_{n-q+1\leq s\leq n+q}\beta^s_{n,q}\left\{\beta^{0}_{s,1}P_{s-1}(x)+
\displaystyle\sum_{s\leq s"\leq s+1}\beta^{s"}_{s,1}P_{s"}(x)\right\}\\
& = & \beta^{0}_{n,q+1}P_{n-(q+1)}(x)+ \displaystyle\sum_{n-q\leq s\leq n+q+1}\beta^s_{n,q+1}P_{s}(x).
\end{eqnarray*}
where $\beta^{0}_{n,q+1}=\beta^0_{n,q}\beta^0_{n-q,1}$ and $\beta^{0}_{n,q+1}=0$ if $n=q+1$ (because $n-q=1$ and then $\beta^0_{n-q,1}=0$).

\textit{Proof of} $(c)$. We show by induction that
\begin{equation}\label{poly eqt}
P_n(x)=\frac{1}{(\V_{\nu}(0))^n}x^n+\displaystyle\sum_{0\leq q\leq n-1}\alpha_q x^q.
\end{equation}
For $n=1$, we have that $P_1(x)=\displaystyle\frac{x}{\V_{\nu}(0)}$.\\
Let $n\in\mathbb{N}^*$ and suppose that \eqref{poly eqt} is true for $n$. The expression
$\frac{P_{n}(x,m)-P_{n}(x,0)}{m}$ is a polynomial in $x$ of degree $\leq n$. On the other hand
\begin{eqnarray*}
\frac{P_{n}(x,m)-P_{n}(x,0)}{m} & = & \frac{\frac{\mathcal{D}_0^nf_{\nu}(x,m)}{f_{\nu}(x,m)}-\frac{\mathcal{D}_0^nf_{\nu}(x,0)}{f_{\nu}(x,0)}}{m}
=\mathcal{D}_0\left(\frac{\mathcal{D}_0^nf_{\nu}(x,m)}{f_{\nu}(x,m)}\right)\\
& = & \frac{f_{\nu}(x,m) \mathcal{D}_0^{n+1}f_{\nu}(x,m)-\mathcal{D}_0^nf_{\nu}(x,m)\mathcal{D}_0f_{\nu}(x,m) }{f_{\nu}(x,m) f_{\nu}(x,0) }\\
& = &  \mathcal{D}_0^{n+1}f_{\nu}(x,m)-\frac{(x-m)}{\V_{\nu}(m)}\mathcal{D}_0^nf_{\nu}(x,m).
\end{eqnarray*}
This implies that \begin{equation}\label{poly eqt2}
\displaystyle\lim_{m\longrightarrow 0}\frac{P_{n}(x,m)-P_{n}(x,0)}{m}=P_{n+1}(x)-\frac{x}{\V_{\nu}(0)}P_{n}(x)
\end{equation}
is well defined. Since the left sided part of \eqref{poly eqt2} can be written as $\displaystyle\sum_{0\leq k\leq n}a_kx^k$, with $a_k\in\mathbb{R}$ for $0\leq k\leq n$, we have,
$$
P_{n+1}(x)
=  \displaystyle\sum_{0\leq k\leq n}a_kx^k+\frac{x}{\V_{\nu}(0)}\left(\frac{x^n}{(\V_{\nu}(0))^n}+\displaystyle\sum_{0\leq q\leq n-1}\alpha_q x^q\right)
=  \frac{x^{n+1}}{(\V_{\nu}(0))^{n+1}}+\displaystyle\sum_{0\leq i\leq n}\delta_i x^i,
$$
with  $\delta_i\in\mathbb{R}$ for $0\leq i\leq n$.
\end{proof}

We give, for each of the six type of \CSK families with polynomial variance function of degree $\leq 2$, the sequence of $\nu$-orthogonal polynomials $P_n(x)$ defined by its recurrence relation for $m_0(\nu)=0$.

\medskip

\begin{tabular}{|c|c|}
  \hline
 type &  Induction relations \\ \hline
  Semi-circle distribution  &   $P_0(x)=1, P_1(x)=x$,  \\
  $\V_{\nu}(m)=1$ &  $P_{n+1}(x)= xP_n(x)-P_{n-1}(x)$,  \\
 &   $  n\geq 1$ \\ \hline
   Marchenko-Pastur &   $P_0(x)=1, P_1(x)=x$, \\
     $\V_{\nu}(m)=1+am$ &   $P_{n+1}(x)= (x-a)P_n(x)-P_{n-1}(x),$  \\
    with $a\neq 0$&    $ \ n\geq 1$ \\ \hline
  free Pascal &  $P_0(x)=1,\ P_1(x)=x,$   \\
   $\V_{\nu}(m)=1+am+bm^2$&    $P_{n+1}(x)= (x-a)P_n(x)-(1+b)P_{n-1}(x), $ \\
 with $b> 0$ and $a^2>4b$&   $ n\geq 1$ \\ \hline
free Gamma  &   $P_0(x)=1,\ P_1(x)=x,$ \\
 $\V_{\nu}(m)=1+am+bm^2$&   $P_{n+1}(x)= (x-a)P_n(x)-(1+b)P_{n-1}(x),$ \\
 with $b> 0$ and $a^2=4b$ &   $ n\geq 1$ \\ \hline
  the free analog of hyperbolic type law &   $P_0(x)=1,\ P_1(x)=x,$ \\
  $\V_{\nu}(m)=1+am+bm^2$ &  $P_{n+1}(x)= (x-a)P_n(x)-(1+b)P_{n-1}(x), $ \\
   with $b> 0$ and $a^2<4b$ &    $ n\geq 1$ \\ \hline
  free binomial  &   $P_0(x)=1,\ P_1(x)=x, $ \\
 $\V_{\nu}(m)=1+am+bm^2$  &   $P_{n+1}(x)= (x-a)P_n(x)-(1+b)P_{n-1}(x), $  \\
 with $-1\leq b< 0$ &   $ n\geq 1$  \\
  \hline
\end{tabular}

\medskip
\subsection{Characterization of the quadratic CSK families in the Meixner way}
Meixner (\cite{Meixner}) characterizes the distributions $\mu$ for which there exists a family of $\mu$-orthogonal polynomials with an exponential generating function. These distributions generate exactly the Morris class of NEFs, more precisely, (see \cite{Letac}):
\begin{theorem}\label{TH 02}
Let $F$ be a NEF on $\mathbb{R}$ and let $\mu$ an element of $F$ with mean $m_0$. If $(H_n)_{n\in\mathbb{N}}$ is a sequence of $\mu$-orthogonal polynomials such that $H_n$ is monic of degree $n$. Then the following statements are equivalent:
\begin{itemize}
\item[(i)] there exist an open set $O$ of $\mathbb{R}$ and $a,\ b:O\longrightarrow\mathbb{R}$ two analytic functions such that for any $z$ in $O$
$$\sum\frac{z^n}{n!}H_n=\exp\{a(z)x+b(z)\}$$
\item[(ii)] $F$ is a quadratic NEF. In this case  $a(z)=\phi_{\mu}(\alpha z)$ and $b(z)=-K_{\mu}(a(z))$ for some real number $\alpha$.
\end{itemize}
The sequence $(H_n)$ is said to have an exponential generating function.
\end{theorem}
We say that the generating function of the sequence of polynomials $T_n$ is given by a Cauchy-Stieltjes type kernel if
\begin{equation}\label{csk type}
\displaystyle\sum_{n\in\mathbb{N}}T_n(x)z^n=\frac{1}{u(z)[f(z)-x]}
\end{equation}
where $u$ and $z\longmapsto zf(z)$ are analytic functions around $0$ with $\displaystyle\lim_{z\longrightarrow 0}\frac{u(z)}{z}=\displaystyle\lim_{z\longrightarrow 0}zf(z)=1$.
The free analog of the Meixner result basing on the notion of orthogonal polynomials is due to Anshelevich (see \cite{Anshelevich}). He characterizes the distributions $\mu$ for which there exists a family of $\mu$-orthogonal polynomials with  a Cauchy-Stieltjes generating function. This distributions turns to be the free Meixner distributions. On the other hand, another proof of the Meixner classification was given in \cite{Kubo}, via orthogonal polynomials using Asai-Kuo-Kubo's criterion basing on the multiplicative renormalization method applied with $\exp(x)$. In \cite{Bozejko-Demni1}, Bozejko and Demni give the free probabilistic interpretation of the multiplicative renormalization method applied with $(1-x)^{-1}$. They give a proof of the Kubo's result on the characterization of the family of probability measure with Cauchy-Stieltjes type generating function for orthogonal polynomials. They have used this result to deduce Bryc and Ismail characterization of quadratic \CSK families, (see \cite{Bozejko-Demni2}).
 Our approach to the characterization of quadratic \CSK families is different from the one given in \cite{Bozejko-Demni2}, and it consist as a first step to give the connection between the $\nu$-orthogonal polynomials, having Cauchy-Stieltjes type kernel generating function, and the polynomials obtained from the density $f_{\nu}(x,m)$ of $Q_m$ with respect to $\nu$.
\begin{theorem}\label{TH 2}
Let ${\mathcal{K}}(\nu)$ be the \CSK family generated by a compactly supported probability measure $\nu$ with mean $m_0(\nu)=0$. Suppose that $(T_n)_{n\in\mathbb{N}}$ is a family of $\nu-$orthogonal polynomials such that $T_n$ is of degree $n$. Then the generating function of $(T_n)_{n\in\mathbb{N}}$ is given by a Cauchy-Stieltjes type kernel as in \eqref{csk type}, if and only if there exists $t\in\mathbb{R}^*$ such that, for all $n\in\mathbb{N}$,
$$T_n(x)=t^nP_n(x),$$
where $(P_n)$ is defined by \eqref{def poly}.
In this case, $f(z)=1/\psi_{\nu}(tz)$ and $u(z)=G_{\nu}(1/\psi_{\nu}(tz))$.
\end{theorem}

\begin{proof}
Up to $\widetilde{T}=T_n/T_0$, we can suppose $T_0=1.$

$\Leftarrow$ Is obvious.

$\Rightarrow$ There exist $r>0$ such that, for all $z\in]-r,r[$,
$$\int\left(\sum_{n\in\mathbb{N}}T_n(x)z^n\right)\nu(dx)=
\int\left(\sum_{n\in\mathbb{N}}T_n(x)T_0(x)z^n\right)\nu(dx)=\int T_0(x)^2\nu(dx)=1.$$
On the other hand, writing the generating function of $(T_n)$ as in \eqref{csk type}, we have
$$\int\left(\sum_{n\in\mathbb{N}}T_n(x)z^n\right)\nu(dx)=\int\frac{1}{u(z)[f(z)-x]}\nu(dx)=\frac{1}{u(z)}G_{\nu}(f(z)).$$
Hence
\begin{equation}\label{G nu (f)}
u(z)=G_{\nu}(f(z)).
\end{equation}
Proceeding similarly, we have that
\begin{equation}\label{TnT1}
\int\left(\sum_{n\in\mathbb{N}}T_n(x)T_1(x)z^n\right)\nu(dx)=\left(\int T_1(x)^2\nu(dx)\right)z.
\end{equation}
Or $T_1$ is a polynomial of degree $1$ in $x$, then there exists $\alpha\in\mathbb{R}^*$ and $\beta\in\mathbb{R}$ such that
\begin{equation}\label{T1}
T_1(x)=\alpha x+\beta.
\end{equation}
Since $\int T_1(x)T_0(x)\nu(dx)=\int T_1(x)\nu(dx)=0$, we get $\beta=0$ and $\int T_1(x)^2\nu(dx)=\int \alpha^2x^2\nu(dx)=\alpha^2 \V_{\nu}(0).$
Furthermore, using \eqref{G nu (f)}-\eqref{T1}, we get
\begin{eqnarray*}
\left(\displaystyle\int T_1(x)^2\nu(dx)\right)z
& = & \displaystyle\int\displaystyle\frac{1}{u(z)[f(z)-x]}T_1(x)\nu(dx) =  \displaystyle\int\displaystyle\frac{\alpha x}{G_{\nu}(f(z))[f(z)-x]}\nu(dx)\\
& = & \displaystyle\int\displaystyle\frac{ \alpha x}{M_{\nu}(1/f(z))[1-x/f(z)]}\nu(dx) =  \displaystyle\int \alpha xP_{1/f(z)}(dx)=\alpha k_{\nu}(1/f(z)),
\end{eqnarray*}
and we deduce that $\alpha^2 \V_{\nu}(0)z=\alpha k_{\nu}(1/f(z)).$
Therefore, with $t=\alpha \V_{\nu}(0)$, we have that $f(z)=1/\psi_{\nu}(tz)$. Finally, we obtain
$$\sum_{n\in\mathbb{N}}T_n(x)z^n=\frac{1}{G_{\nu}(1/\psi_{\nu}(tz))[1/\psi_{\nu}(tz)-x]}=f_{\nu}(x,tz).$$

\end{proof}
The following result is the CSK-version of Theorem \ref{TH 02}.
\begin{corollary}
Let ${\mathcal{K}}(\nu)$ be the \CSK family generated by a compactly supported probability measure $\nu$ with mean $m_0(\nu)=0$. Then there exists a family of $\nu$-orthogonal polynomials with a Cauchy-Stieltjes type kernel generating function if and only if ${\mathcal{K}}(\nu)$ is quadratic.
\end{corollary}
\begin{proof}
Follows easily from Theorems \ref{TH 1} and \ref{TH 2}.
\end{proof}
\subsection*{Acknowledgement}
We thank  Pr. Abdelhamid Hassairi for his suggestion to work on the polynomials associated with Cauchy-Stieltjes kernel families. We also thank Pr. W{\l}odzimierz  Bryc for his important comments that help to improve the paper.

\end{document}